\documentclass[12pt]{article}
\usepackage[nottoc]{tocbibind}
\usepackage{mathtools,commath}
\usepackage{amssymb,amsmath,amsthm}
\usepackage{enumitem}
\usepackage{multirow}
\usepackage{bbm}
\usepackage{hyperref}
\usepackage{fullpage}
\newtheorem{thm}{Theorem}
\newtheorem{lemma}[thm]{Lemma}

\newtheorem{cor}{Corollary}
\newtheorem*{definition}{Definition}

\newtheorem{theorem}{Theorem}[section]

\def\l{\lambda}

\def\s{\sigma}
\def\r{\rho}

\def\bR{\mathbb R}

\def\DMO{\DeclareMathOperator}

\DMO{\GL}{GL}
\DMO{\SL}{SL}
\DMO{\gl}{\mathfrak{gl}}
\DMO{\ssl}{\mathfrak{sl}}

\usepackage{subfiles}

\title{Existence and boundary behaviour of radial solutions for weighted elliptic systems with gradient terms}

\author{ Daniel Devine\footnote{School of Mathematics, Trinity College Dublin; {\tt dadevine@tcd.ie}} 
    $\;\;$        $\;$
Gurpreet Singh\footnote{School of Mathematical Sciences, Dublin City University, Ireland; {\tt gurpreet.bajwa2506@gmail.com}}}
\date{}

\begin{document}
\maketitle

\begin{abstract}
We are concerned with the existence and boundary behaviour of positive radial solutions for the system
\begin{equation*}
\left\{
\begin{aligned}
\Delta u&=|x|^{a}v^{p}  &&\quad\mbox{ in } \Omega, \\
\Delta v&=|x|^{b}v^{q}f(|\nabla u|) &&\quad\mbox{ in } \Omega,
\end{aligned}
\right.
\end{equation*}
where $\Omega \subset \bR^N$ is either a ball centered at the origin or the whole space $\bR^N$, $a$, $b$, $p$, $q> 0$, and $f \in C^1[0, \infty)$ is an increasing function such that $f(t)> 0$ for all $t> 0$. Firstly, we study the existence of positive radial solutions in case when the system is posed in a ball corresponding to their behaviour at the boundary. Next, we take $f(t) = t^s$, $s> 1$, $\Omega = \bR^N$ and by the use of dynamical system techniques we are able to describe the behaviour at infinity of such positive radial solutions.

\end{abstract}

\smallskip


\smallskip


\section{Introduction}
In this paper, we study the positive radial solutions for the weighted semilinear elliptic system 
\begin{equation}\label{system}
\left\{
\begin{aligned}
\Delta u&=|x|^{a}v^{p}  &&\quad\mbox{ in } \Omega, \\
\Delta v&=|x|^{b}v^{q}f(|\nabla u|) &&\quad\mbox{ in } \Omega,
\end{aligned}
\right.
\end{equation}
where $\Omega \subset \bR^N$ is either a ball centered at the origin with radius $R> 0$ or the whole space $\bR^N$, $a$, $b$, $p$, $q> 0$ and $f \in C^1[0, \infty)$ is an increasing function such that $f(t)> 0$ for all $t> 0$. We start with the assumption that $u$ and $v$ are positive radially symmetric solutions of \eqref{system}. There is no prior condition at the boundary for $u$ and $v$, but it will be required as we move forward in our analysis as we are concerned with the classification of all the solutions of \eqref{system}. In \cite{S2015}, Singh had studied the system \eqref{system} in case when $a = b = q = 0$, that is, the system 
\begin{equation}\label{gsys}
\left\{
\begin{aligned}
\Delta u&=v^{p}  &&\quad\mbox{ in } \Omega, \\
\Delta v&=f(|\nabla u|) &&\quad\mbox{ in } \Omega.
\end{aligned}
\right.
\end{equation}
The author considered the system \eqref{gsys} both in case of a ball centered at the origin with positive radius and in the whole space $\bR^N$ and classified all the positive radial solutions of \eqref{gsys} together with the behaviour of solutions at the boundary. In case $\Omega$ is a ball, system \eqref{gsys} with $p=1$ and $f(t) = t^2$, that is,
\begin{equation}\label{gsys1}
\left\{
\begin{aligned}
\Delta u&=v &&\quad\mbox{ in } \Omega, \\
\Delta v&=|\nabla u|^2 &&\quad\mbox{ in } \Omega,
\end{aligned}
\right.
\end{equation}
was first studied by Diaz, Lazzo, and Schmidt in \cite{DLS2005}. This special choice of $p$ and function $f$ appears in the study of the dynamics of viscous, heat-conducting fluid. In \cite{DLS2005}, the authors obtained that the system \eqref{gsys1} has one positive solution which blows up at the boundary and the authors also observed that in case of small dimensions ($N\leq 9$), there exists one sign-changing solution which also blows up at the boundary. In \cite{DRS2007, DRS2008}, Diaz, Rakotoson, and Schmidt extended these results to time dependent systems.

\smallskip

The problems related to boundary blow-up solutions can be linked back to atleast a century ago, for instance, Bieberbach \cite{B1916} in 1916 studied the boundary blow-up solutions for the equation $\Delta u=e^u$ in a planar domain. From that point of time onwards, there has been many new techniques which have been devised to tackle such solutions ( see \cite{GRbook2008,GRbook2012,Rbk}). In the last few decades, semilinear elliptic equations with nonlinear gradient terms have been studied extensively for such boundary blow-up solutions ( see \cite{AGMQ2012,CPW2013,FQS2013, FV2017, GNR2002, MMMR2011} ).

In \cite{GGS2019}, Ghergu, Giacomoni and Singh studied the following more general quasilinear elliptic system with nonlinear gradient terms 

\begin{equation}\label{gsys2}
\left\{
\begin{aligned}
\Delta_{p} u&=v^{m} |\nabla u|^{\alpha} &&\quad\mbox{ in } \Omega, \\
\Delta_{p} v&=v^{\beta}|\nabla u|^q &&\quad\mbox{ in } \Omega,
\end{aligned}
\right.
\end{equation}
in which the authors classified all the positive radial solutions in case $\Omega$ is a ball and also obtained the behaviour at infinity of such solutions.

We first consider the case when $\Omega = B_R$ is a ball of radius $R> 0$ and centered at the origin. We obtained that in this case the system \eqref{system} has positive radially symmetric solutions $(u, v)$ such that $u$ or $v$ ( or both ) blow up around $\partial{\Omega}$ if and only if
\begin{equation}\label{gsys3}
\int_{1}^\infty\frac{ds}{\Big(\displaystyle \int_0^s F(t)dt  \Big)^{p/(2p - q+1)}} <\infty \quad\mbox{ where }\; F(t)=\int_0^t f(k)dk.
\end{equation}
We also study the full classification of positive radially symmetric solutions in case the system is posed in a ball. The condtion \eqref{gsys3} is similar to optimal conditions obtained by Keller \cite{K1954} and Osserman \cite{O1957} in 1950s while studying the existence of a solution to the boundary blow-up problem
\begin{equation}\label{gsys4}
\left\{
\begin{aligned}
\Delta u&=f(u)&&\quad\mbox{ in }\Omega,\\
u&=\infty &&\quad\mbox{ on }\partial\Omega,
\end{aligned}
\right.
\end{equation}
where $\Omega\subset \bR^N$ is a bounded domain and $f\in C^{1}[0,\infty)$ is a nonnegative increasing function. The authors obtained that \eqref{gsys4} has $C^2(\Omega)$ solutions if and only if
\begin{equation}\label{gsys5}
\int_1^\infty\frac{ds}{\sqrt{F(s)}}<\infty\quad\mbox{ where }\; F(s)=\int_0^s f(t)dt.
\end{equation}
Equation \eqref{gsys5} has also been seen in various other circumstances as it is related to the maximum principle for nonlinear elliptic inequalities. For example, if $u\in C^2(\Omega)$ is nonnegative and satisfies $\Delta u\leq f(u)$ in $\Omega$, then, if $u$ vanishes at a point in $\Omega$, it must vanish everywhere in $\Omega$. For various extensions of this result, one could refer to Vazquez \cite{V1984} and to Pucci, Serrin and Zou \cite{PS1993,PS2004,PSZ1999}. Next, if $f(t)= t^s$, $s\geq 1$, then in this case we are able to find the exact rate at which the solution $(u, v)$ blows up at the infinity. Here, we have used the dynamical system techniques for cooperative systems with negative divergence to find out the exact blow up rate at infinity. 

\smallskip

\section{Main Results}
First, we consider the system \eqref{system} in case where $\Omega = B_R$, that is, we study the system

\begin{equation}\label{system1}
\left\{
\begin{aligned}
\Delta u&=|x|^{a}v^{p}  &&\quad\mbox{ in } B_R, \\
\Delta v&=|x|^{b}v^{q}f(|\nabla u|) &&\quad\mbox{ in } B_R,
\end{aligned}
\right.
\end{equation}
where $B_R\subset \bR^N,$ $N\geq 2$ is an open ball of radius $R>0$ centred at the origin, $a$, $b$, $p$, $q> 0$ and $f \in C^1[0, \infty)$ is an increasing function such that $f(t)> 0$ for all $t> 0$.

In this paper, sometimes system \eqref{system1} has one of the following boundary conditions:
\begin{itemize}
\item either
$
\mbox{ $u$ and $v$ are bounded in $B_R$ ;}
$
\item or
\begin{equation*}
u\mbox{ is bounded in $B_R$ and } \lim_{|x|\nearrow R}v(x)=\infty ;
\end{equation*}
\item or
\begin{equation*}
\lim_{|x|\nearrow R}u(x)=\lim_{|x|\nearrow R}v(x)=\infty.
\end{equation*}
\end{itemize}

We note that by \eqref{system}, the boundary condition $\lim_{|x|\nearrow R}u(x)=\infty$ and $v$ is bounded is not possible. Throughout this paper, $C$ denotes some positive constant which may vary from line to line, or within the same line, and $r$ will often be used in place of $|x|$.

\begin{theorem}\label{thm1}
Let us suppose that $(u, v)$ is a positive radial solution of \eqref{system1}. Then, we have
\begin{enumerate}
\item[(i)] Both $u$ and $v$ are bounded if and only if
\begin{equation}\label{bounded}
\int_{1}^\infty\frac{ds}{\Big(\displaystyle \int_0^s F(t)dt  \Big)^{p/(2p - q+1)}} = \infty.
\end{equation}
\item[(ii)] $u$ is bounded and $\lim_{|x|\nearrow R}v(r)=\infty$  if and only if
\begin{equation}\label{int1}
\int_{1}^\infty\frac{s ds}{\Big(\displaystyle \int_0^s F(t)dt  \Big)^{p/(2p - q+1)}} <\infty.
\end{equation}
\item[(iii)] $\lim_{|x|\nearrow R}u(x)=\lim_{|x|\nearrow R}v(x)=\infty$ if and only if
\begin{equation}\label{int2}
\int_{1}^{\infty}\frac{ds}{\left(\int_{0}^{s}F(t)dt\right)^{\frac{p}{2p-q+1}}} <\infty\quad\mbox{and} \int_{1}^{\infty}\frac{s\ ds}{\left(\int_{0}^{s}F(t)dt\right)^{\frac{p}{2p-q+1}}}=\infty.
\end{equation}
\end{enumerate}
\end{theorem}

If $f(t) = t^s$, $s\geq 1$, then theorem \ref{thm1} gives the following result:

\begin{cor}\label{cor1}
Consider the system
\begin{equation}\label{system2}
\left\{
\begin{aligned}
\triangle u&=|x|^av^p &&\quad\mbox{ in }\ B_{R}, \\
\triangle v&=|x|^{b}v^q|\nabla u|^{s} &&\quad\mbox{ in }\ B_{R}.
\end{aligned}
\right.
\end{equation}
 We have:
\begin{enumerate}[label=(\roman*)]
\item All positive radial solutions to \eqref{system2} are bounded if and only if
\begin{equation*}
ps+q-1\leq 0.
\end{equation*}
\item There exists positive radial solutions to \eqref{system2} satisfying $u$ bounded and $\lim_{|x|\nearrow R}v(r)=\infty$ if and only if
\begin{equation*}
s>2\left(1+\frac{1-q}{p}\right).
\end{equation*}
\item There exists positive radial solutions to \eqref{system2} satisfying $\lim_{|x|\nearrow R}u(x)=\lim_{|x|\nearrow R}v(x)=\infty$ if and only if
\begin{equation*}
\frac{1-q}{p}<s\leq2\left(1+\frac{1-q}{p}\right).
\end{equation*}
\end{enumerate}
\end{cor}

\smallskip

Now, we study the system \eqref{system2} in the whole space $\bR^N$, that is,

\begin{equation}\label{sys1}
\left\{
\begin{aligned}
\Delta u&= |x|^{a}v^p&&\quad\mbox{ in } \bR^N,\\
\Delta v&= |x|^{b}v^q |\nabla u|^s &&\quad\mbox{ in } \bR^N,
\end{aligned}
\right.
\end{equation}
where $a$, $b$, $p$, $q> 0$ and $s \geq 1$.

\begin{theorem}\label{thm2}
System \eqref{sys1} has positive radial solutions if and only if
\begin{equation}\label{int sqrt f inf}
ps+q-1\leq 0.
\end{equation}
\end{theorem}

\smallskip

\begin{theorem}\label{thmrn}
Assume $p< 1$ and $ps + q< 1$. Let $(u,v)$ be a positive radially symmetric solution of \eqref{sys1}. If 
\begin{equation}\label{div}
\frac{p(s-2)(s+as+b+2)}{1-ps-q} \leq 2(N+a-1),
\end{equation}
then
\begin{equation*}
\lim_{|x|\rightarrow \infty}\frac{u(x)}{|x|^{\frac{(a+2)(1-ps-q)+ ps(a+1)+bp+2q}{1-ps-q}}}= \frac{(AB^s K)^{\frac{p}{ps + q - 1}}}{DK}
\end{equation*}
and
\begin{equation*}
\lim_{|x|\rightarrow \infty}\frac{v(x)}{|x|^{\frac{(a+1)s + b + 2}{1-ps-q}}}= (AB^s K)^{\frac{1}{ps + q - 1}},
\end{equation*}
where, 
$$
A = 2 + \frac{b+2q+s(1+a+2p)}{1-ps-q},
$$
	
$$
B = N+a + p\Big(2+\frac{b+2q+s(1+a+2p)}{1-ps-q}\Big),
$$
	
$$
K = N + \frac{b+2q+s(1+a+2p)}{1-ps-q},
$$
	
$$
D = 2+a + p\Big(2+\frac{b+2q+s(1+a+2p)}{1-ps-q}\Big).
$$
\end{theorem}

\section{Notes on dynamical systems}
Let $x=(x_1,x_2,x_3)$, $y=(y_1,y_2,y_3)$ be any two points in $\bR^{3}$, then for $x_i\leq y_i$ we write
$$
x\leq y 
$$
where $i=1,2,3$. For $x\leq y$ and $x\neq y$, we write $x<y$.

Also, the open ordered interval is defined as
$$
[[x,y]]=\{z\in \bR^3:x<z<y\}\subset \bR^3.
$$
Consider the initial value problem
\begin{equation}\label{det1}
\left\{
\begin{aligned}
&\xi_{t}=h(\xi) \quad\mbox{ for } t\in \bR,\\
&\xi(0)=\xi_{0},
\end{aligned}
\right.
\end{equation}
where $ h:\bR^{3}\rightarrow \bR$ is a $C^{1}$ function. This implies that there exists a unique solution $\xi$ of \eqref{det1} defined in a maximal time interval for any $\xi_{0}\in \bR^{3}$.  Let $\varphi(\cdot,\xi_{0})$ denotes the flow associated to \eqref{det1}, that is,  $t\longmapsto \varphi(t,\xi_{0})$ is the unique solution of \eqref{det1} defined in a maximal time interval. Let us suppose that the vector field $h$ is cooperative, that is
$$
\frac{\partial h_{i}}{\partial x_{j}}\geq 0 \quad \mbox{ for } 1\leq i,j\leq 3,\;\; i\neq j.
$$
Next, follow the results due to Hirsch \cite{Hirsch1989, Hirsch1990}.
\begin{theorem}\label{thmdet1}{\rm (see \cite[Theorem 1]{Hirsch1990})}
Any compact limit set of \eqref{det1} contains an equilibrium or is a cycle.
\end{theorem}
\begin{definition}
A finite sequence of equilibria $\zeta_{1},\zeta_{2},\dots,\zeta_{n}=\zeta_{1}$, $(n\geq 2)$ is known as a circuit such that $W^{u}(\xi_{i})\cap W^{s}(\xi_{i+1})$ is non-empty. Here, $W^{u}$ and $W^{s}$ respresents the stable and unstable manifolds respectively.
\end{definition}

\medskip
{\bf{Note: }}	There is no circuit in case all the equilibria are hyperbolic and also their stable and unstable manifolds are mutually transverse.

\begin{theorem}\label{thmdet2}{\rm (see \cite[Theorem 2]{Hirsch1990}). }
Let us assume that $L\subset \bR^{3}$ is a compact set such that:
\begin{enumerate}
\item [ (i) ] There is no circuit and all the equilibria in $L$ are hyperbolic.
\item [ (ii)] The number of cycles in $L$ which have period less than or equal to $T$ is finite, where $T>0$.
\end{enumerate}
Then:
\begin{enumerate}
\item [ (a) ] Every limit set in $L$ is an equilibrium or cycle.
\item [ (b) ] $L$ has finite number of cycles.
\end{enumerate}
	
\end{theorem}

\begin{theorem}\label{thmdet3}{\rm (see \cite[Theorem 7]{Hirsch1989})}
Assume that $\xi_{1}$, $\xi_{2}\in \bR^{3}$ such that $\xi_1< \xi_2$. Further, if
$$
{\rm div}{h}<0 \quad\mbox{ in } [[\xi_1,\xi_2]],
$$
then there are no cycles of \eqref{det1} in $[[\xi_1,\xi_2]].$
\end{theorem}

\bigskip

\section{ Proof of Theorem \ref{thm1}}
We divide our proof into the following two lemmas:

\begin{lemma} 
Let $(u, v)$ be a positive radial solution of \eqref{system1}. Then
\begin{equation*}
 \lim_{|x|\nearrow R}v(x)=\infty
\end{equation*}
if and only if 
\begin{equation*}
    \int_{1}^{\infty}\frac{ds}{\left(\int_{0}^{s}\sqrt{f(t)}dt\right)^{\frac{2p}{2p-q+1}}} <\infty.
\end{equation*}
\end{lemma}
\begin{proof}
First, assume that $\lim_{r\to R}v(r)=\infty$. System \eqref{system} can be rewritten as
\begin{equation*}
	\left\{
    \begin{aligned}
    &(u'(r)r^{N-1})'= r^{N+a-1}v^{p}(r),\\
    &(v'(r)r^{N-1})'= r^{N+b-1}v^{q}(r)f(|\nabla u|),\\
    &u'(0)=v'(0)=0.
    \end{aligned}
\right.
\end{equation*}
An integration over $(0,r)$, $0<r<R$, then gives us
\begin{equation*}
	\left\{
    \begin{aligned}
    u'(r)&= r^{1-N}\int_{0}^{t}t^{N+a-1}v^{p}(t)dt,\\
    v'(r)&=r^{1-N}\int_{0}^{t}t^{N+b-1}v^{q}(t)f(|\nabla u|)dt,
    \end{aligned}
\right.
\end{equation*}
from which it follows that $u$ and $v$ are increasing in $(0,R)$. We can thus take the first integral equation above and estimate $u'$ as 
\begin{equation*}
    u'(r)\leq r^{1-N}v^{p}(r)\int_{0}^{r}t^{N+a-1}dt=\frac{r^{a+1}v^{p}(r)}{N+a}
\end{equation*}
so from \eqref{system} we have that
\begin{equation*}
r^{a}v^{p}(r)\leq u''(r)+\frac{N-1}{r}\cdot \frac{r^{a+1}v^{p}(r)}{N+a},
\end{equation*}
which further implies,
\begin{equation*}
\frac{1+a}{N+a}r^{a}v^{p}(r) \leq u''(r).
\end{equation*}
Similarly we have
\begin{equation*}
    v'(r)\leq \frac{r^{b}v^{q}(r)f(u'(r))}{N+b},
\end{equation*}
from which it follows
\begin{equation*}
    \frac{1+b}{N+b}r^{b}v^{q}(r)f(u'(r)) \leq v''(r).
\end{equation*}
We thus have the following two estimates for all $0<r<R$, which will be important as we proceed:
\begin{equation}\label{u'' bound}
\frac{1+a}{N+a}r^{a}v^{p}(r)\leq u'' \leq r^{a}v^{p}(r), 
\end{equation}

\begin{equation}\label{v'' bound}
 \frac{1+b}{N+b}r^{b}v^{q}(r)f(u'(r))\leq v'' \leq r^{b}v^{q}(r)f(u'(r)). 
 \end{equation}
Now, where $w=u'$, we multiply the right inequality in \eqref{v'' bound} by $v'$ and integrate over $[0,r]$ to find
\begin{equation}\label{v' bound}
\begin{aligned}
 \frac{v'(r)^{2}}{2} &\leq \int_{0}^{r} t^{b}v^{q}(t)v'(t)f(w(t)) dt \\
 &\leq r^{b}f(w(r))\int_{0}^{r}\left(\frac{v^{q+1}(t)}{q+1}\right)'dt \\
   &\leq Cf(w(r))v^{q+1}(r)\quad\mbox{for all}\ 0<r<R,
\end{aligned}
\end{equation}
from which it follows
\begin{equation}\label{root f bound}
    v^{-\frac{q+1}{2}}(r)v'(r)\leq C\sqrt{f(w(r))}.
\end{equation}
Now, fix $0<r_0<R$. From \eqref{u'' bound}, we see that 
for all $r>r_0$ we have
\begin{equation*}
    v^{p}(r)\leq Cw'(r)
\end{equation*}
and multiplying this inequality by \eqref{root f bound} we obtain
\begin{equation*}
    v^{p-\frac{q+1}{2}}(r)v'(r)\leq C\sqrt{f(w(r))}w'(r).
\end{equation*}
Integrating both sides of this equation over $(r_0,r)$ gives
\begin{equation*}
\begin{aligned}
    v^{p+\frac{1-q}{2}}(r)-v^{p+\frac{1-q}{2}}(r_0)&\leq C\int_{r_0}^{r}w'(t)\sqrt{f(w(t))}dt\\
    &\leq C\int_{w(r_0)>0}^{w(r)}\sqrt{f(t)}dt\\
    &\leq C\int_{w(0)=0}^{w(r)}\sqrt{f(t)}dt\quad\mbox{for all}\ r>r_0.
\end{aligned}
\end{equation*}
Now, since $\lim\limits_{r\to R}v(r)=\infty$, we have that there exists $\r<R$ such that
\begin{equation*}
    v^{p+\frac{1-q}{2}}(r)=\left(v^{p}(r)\right)^{\frac{2p-q+1}{2p}}\leq C  \int_{0}^{w(r)}\sqrt{f(t)}dt\quad\mbox{for all}\ \r<r.
\end{equation*}
Let $\tilde{r}=\mbox{max}\{r_{0},\r\}$. Using \eqref{u'' bound} again we see
\begin{equation*}
    (w'(r))^{\frac{2p-q+1}{2p}}\leq C  \int_{0}^{w(r)}\sqrt{f(t)}dt\quad\mbox{for all}\ \tilde{r}<r,
\end{equation*}
from which it follows that
\begin{equation*}
    \frac{w'(r)}{\left(\int_{0}^{w(r)}\sqrt{f(t)}dt\right)^{\frac{2p}{2p-q+1}}}\leq C \quad\mbox{for all}\ \tilde{r}<r.
\end{equation*}
We can integrate both sides of this expression over $(\tilde{r},r)$ to find
\begin{equation*}
    \int_{\tilde{r}}^{r}\frac{w'(t)dt}{\left(\int_{0}^{w(t)}\sqrt{f(s)}ds\right)^{\frac{2p}{2p-q+1}}}\leq C \int_{\tilde{r}}^{r}dt
\end{equation*}
or, after a change of variables
\begin{equation*}
    \int_{w(\tilde{r})}^{w(r)}\frac{ds}{\left(\int_{0}^{s}\sqrt{f(t)}dt\right)^{\frac{2p}{2p-q+1}}}\leq C (r-\tilde{r})\leq Cr.
\end{equation*}
Letting $r\to R$, we see
\begin{equation*}
    \int_{w(\tilde{r})}^{\infty}\frac{ds}{\left(\int_{0}^{s}\sqrt{f(t)}dt\right)^{\frac{2p}{2p-q+1}}}\leq C R<\infty.
\end{equation*}
Hence,
\begin{equation}\label{int sqrt f}
    \int_{1}^{\infty}\frac{ds}{\left(\int_{0}^{s}\sqrt{f(t)}dt\right)^{\frac{2p}{2p-q+1}}}<\infty.
\end{equation}
A minor adjustment of Lemma 4.1 in \cite{S2015} gives us that \eqref{int sqrt f} is equivalent to 
\begin{equation}\label{F integral}
    \int_{1}^{\infty}\frac{ds}{\left(\int_{0}^{s}F(t)dt\right)^{\frac{p}{2p-q+1}}} <\infty.
\end{equation}
Assume now that \eqref{F integral} holds. We see that \eqref{system} can be rewritten as
\begin{equation}\label{new system}
\left\{
    \begin{aligned}
    u(r)&=u(0)+\int_{0}^{r}t^{1-N}\left(\int_{0}^{t}s^{a}v^{p}(s)ds\right)dt, &&\quad r>0,\\
    v(r)&=v(0)+\int_{0}^{r}t^{1-N}\left(\int_{0}^{t}s^{b}v^{q}(s)f(|u'(s)|)ds\right)dt, &&\quad r>0,\\
    u(0)&>0, \quad v(0)>0.
    \end{aligned}
\right.
\end{equation}
Using a contraction mapping argument, we can show that system \eqref{new system} has a solution $(u,v)$ defined on some maximum interval $[0,R_{0})$. Now, fix $\r\in(0,R_{0})$, and, recalling \eqref{u'' bound} and \eqref{v'' bound}, we have
\begin{equation}\label{f,w' bounds}
    \begin{aligned}
    f(w(r))\leq Cv^{-q}(r)v''(r)\quad&\mbox{for all}\ \r\leq r<R_{0},\\
    w'(r)\leq Cv^{p}(r)&\mbox{for all}\ \r\leq r<R_{0}.
    \end{aligned}
\end{equation}
Multiplying the two inequalities in \eqref{f,w' bounds} and the integrating over $[\r,r]$ we find
\begin{equation*}
    F(w(r))-F(w(\r))\leq Cv^{p-q}(r)v'(r),
\end{equation*}
which we can express as
\begin{equation*}
    F(w(r))\leq Cv^{p-q}(r)v'(r).
\end{equation*}
 Using \eqref{f,w' bounds} again, we see that the above becomes
\begin{equation*}
    w'(r)F(w(r))\leq Cv^{2p-q}(r)v'(r)\quad\mbox{for all}\ \r\leq r<R_{0}.
\end{equation*}
Define
\begin{equation*}
    G(r)\coloneqq\int_{\r}^{r}F(t)dt\quad\mbox{for all}\ \r\leq r<R_{0}.
\end{equation*}
We thus have
\begin{equation*}
\begin{aligned}
    G(w(r))=\int_{\r}^{w(r)}F(t)dt&\leq C\int_{\r}^{r}v^{2p-q}(t)v'(t)dt\\
    &\leq C[v^{p}(r)]^{\frac{2p-q+1}{p}}\\
    &\leq C[w'(r)]^{\frac{2p-q+1}{p}}\quad\mbox{for all}\ \r\leq r<R_{0}.
\end{aligned}
\end{equation*}
Hence
\begin{equation*}
    C\leq \frac{w'(r)}{G(w(r))^{\frac{p}{2p-q+1}}}\quad\mbox{for all}\ \r\leq r<R_{0}.
\end{equation*}
Integrating the above over $[\r,r]$ gives
\begin{equation*}
    C(r-\r)\leq \int_{\r}^{r}\frac{w'(t)}{G(w(t))^{\frac{p}{2p-q+1}}}dt=\int_{w(\r)}^{w(r)}\frac{dt}{G(t)^{\frac{p}{2p-q+1}}}
\end{equation*}
Now, letting $r\to R_{0}$ we see that
\begin{equation*}
    C(R_{0}-\r)\leq C\int_{1}^{\infty}\frac{dt}{\phantom{ab}G(t)^{\frac{p}{2p-q+1}}} 
\end{equation*}
from which it follows that
\begin{equation*}
    R_{0}\leq C\int_{1}^{\infty}\frac{dt}{\left(\int_{0}^{s}F(t)dt\right)^{\frac{p}{2p-q+1}}} <\infty.
\end{equation*}
We have obtained a positive radial solution $(u,v)$ of \eqref{system} in $B_{R_{0}}$ satisfying $\lim_{r\to R_{0}}v(r)=\infty$. Now, if $R>0$ is any arbitrary radius, we set
\begin{equation*}     \tilde{f}(t)=\l^{b+\frac{a(1-q)}{p}+2(1+\frac{1-q}{p})}f\left(\frac{t}{\l}\right)\quad\mbox{for all}\ t\geq 0.
\end{equation*}
By the above, we know there exists $(\tilde{u},\tilde{v})$ satisfying
\begin{equation*}
	\left\{
    \begin{aligned}
    \triangle \tilde{u}&=r^{a}\tilde{v}^{p}\quad&\mbox{in}\ B_{R_{0}},\\
    \triangle \tilde{v}&=r^{b}\tilde{v}^{q} \tilde{f}(|\nabla \tilde{u}|)\quad&\mbox{in}\ B_{R_{0}},
    \end{aligned}
\right.
\end{equation*}
where $B_{R_{0}}$ is a maximum ball of existence. Let
\begin{equation*}
	\left\{
    \begin{aligned}
    u(r)&=\tilde{u}\left(\frac{r}{\l}\right)\quad&\mbox{in}\ B_{R},\\
    v(r)&=\l^{-\frac{a+2}{p}}\tilde{v}\left(\frac{r}{\l}\right)\quad&\mbox{in}\ B_{R}.
    \end{aligned}
 \right.
\end{equation*}
Taking $\l=\frac{R}{R_{0}}$ we see that $(u,v)$ is a solution to \eqref{system} in $B_{R}$.
\end{proof}
\begin{lemma}
System \eqref{system1} has a positive radial solution satisfying
\begin{equation}\label{u bound v not}
    u\ \mbox{is bounded in}\ B_{R}\ \mbox{and}\ \lim\limits_{|x|\nearrow R}v(x)=\infty
\end{equation}
if and only if
\begin{equation*}
    \int_{1}^{\infty}\frac{s\ ds}{\left(\int_{0}^{s}F(t)dt\right)^{\frac{p}{2p-q+1}}} <\infty.
\end{equation*}
Also, system \eqref{system1} has a positive radial solution satisfying
\begin{equation}\label{u,v unbound}
   \lim\limits_{|x|\nearrow R}u(x)= \lim\limits_{|x|\nearrow R}v(x)=\infty
\end{equation}
if and only if
\begin{equation*}
\int_{1}^{\infty}\frac{ds}{\left(\int_{0}^{s}F(t)dt\right)^{\frac{p}{2p-q+1}}} <\infty\quad\mbox{and} \int_{1}^{\infty}\frac{s\ ds}{\left(\int_{0}^{s}F(t)dt\right)^{\frac{p}{2p-q+1}}}=\infty.
\end{equation*}
\end{lemma}
\begin{proof}
It is enough to prove this for a solution satisfying \eqref{u bound v not}, so assume $(u,v)$ is a solution to \eqref{system} satisfying \eqref{u bound v not}. From Lemma 1 we know that $f$ must satisfy \eqref{F integral}. From Lemma 4.1 in ~\cite{S2015}, we know that
\begin{equation*}
    \left(\int_{0}^{2s}F(t)dt\right)^{\frac{p}{2p-q+1}}\geq \left(\int_{0}^{s}\sqrt{f(t)}dt\right)^{\frac{2p}{2p-q+1}}\quad\mbox{for all}\ s\geq 0.
\end{equation*}
Arguing in a similar way to Lemma 1, we find that there exists $\r\in(0,R)$ such that
\begin{equation}\label{int comparison}
    \int_{w(r)}^{\infty}\frac{ ds}{\left(\int_{0}^{2s}F(t)dt\right)^{\frac{p}{2p-q+1}}}\leq \int_{w(r)}^{\infty}\frac{ ds}{\left(\int_{0}^{s}\sqrt{f(t)}dt\right)^{\frac{2p}{2p-q+1}}}\leq C_{1}(R-r)
\end{equation}
and
\begin{equation}\label{int lower bound}
    \int_{w(r)}^{\infty}\frac{ ds}{\left(\int_{0}^{s}F(t)dt\right)^{\frac{p}{2p-q+1}}}\geq C_{2}(R-r)
\end{equation}
for all $\r<r<R$. Now, let $\Gamma:(0,\infty)\rightarrow (0,\infty)$ be defined as
\begin{equation*}
    \Gamma(t)=\int_{t}^{\infty}\frac{ ds}{\left(\int_{0}^{s}F(\s)d\s\right)^{\frac{p}{2p-q+1}}}.
\end{equation*}
We note that $\Gamma$ is decreasing and by \eqref{F integral} we see $\lim_{t\to \infty}\Gamma(t)=0$. From \eqref{int comparison}and \eqref{int lower bound} we find
\begin{equation*}
    \Gamma(2w(r))\leq C_{1}(R-r)\quad\mbox{and}\quad\ \Gamma(w(r))\geq C_{2}(R-r)\quad\mbox{for all}\ \r\leq r<R. 
\end{equation*}
Since $\Gamma$ is decreasing, this then implies
\begin{equation*}
    \begin{aligned}
    2w(r)\geq \Gamma^{-1}(C_{1}(R-r))\quad&\mbox{for all}\ \r\leq r<R,\\
    w(r)\leq \Gamma^{-1}(C_{2}(R-r))\quad&\mbox{for all}\ \r\leq r<R.
    \end{aligned}
\end{equation*}
Now, recalling that
\begin{equation*}
    u(r)=u(\r)+\int_{\r}^{r}w(t)dt\quad\mbox{for all}\ \r\leq r<R,
\end{equation*}
we see that $\lim_{r\to R}u(r)=\infty$ if and only if
\begin{equation*}
    \int_{\r}^{R}w(t)dt=\infty
\end{equation*}
if and only if
\begin{equation*}
    \int_{\r}^{R}\Gamma^{-1}(C(R-r))dt=\infty,
\end{equation*}
for some $C>0$. Hence $\lim_{r\to R}u(r)=\infty$ if and only if
\begin{equation*}
    \int_{0}^{C(R-\r)}\Gamma^{-1}(u)du=\infty,
\end{equation*}
if and only if
\begin{equation*}
    \int_{0}^{1}\Gamma^{-1}(u)du=\infty.
\end{equation*}
The change of variables $t=\Gamma^{-1}(u)$ then gives us that $\lim_{r\to R}u(r)=\infty$ if and only if
\begin{equation*}
    \int_{1}^{\infty}\frac{s\ ds}{\left(\int_{0}^{s}F(t)dt\right)^{\frac{p}{2p-q+1}}}=\infty.
\end{equation*}
To show that the opposite implication holds, one can proceed as in Lemma 1 to obtain the local existence of a solution, and then use a scaling argument.
\end{proof}

\medskip

\section{ Proof of Theorem \ref{thm2} }

First, we assume that \eqref{int sqrt f inf} holds. We have seen that this implies the existence of a positive, radial solution $(u,v)$ in a maximum ball. Both $u$ and $v$ are increasing, and by the corollary \ref{cor1} we know that $u$ and $v$ are both bounded. Hence the domain of existence must be $\bR^{N}$.
\bigskip

Conversely, assume that \eqref{int sqrt f inf} does not hold, and let $(U,V)$ be a solution to \eqref{sys1}. We know from Lemma 2 that there exists a solution $(\tilde{u},\tilde{v})$ satisfying
\begin{equation*}
\lim\limits_{r\to 1}\tilde{v}(r)=\infty.
\end{equation*}
Now, if $(U,V)$ is a solution to \eqref{sys1} in $\bR^{N}$, then so is $(U_{\l},V_{\l})$, where 
\begin{equation*}
\begin{aligned}
    U_{\l}(r)&=\l^{\frac{p(b+2-s)-(a+2)(q-1)}{sp+q-1}}U(\l r) \\
    V_{\l}(r)&=\l^{\frac{b+2+s(a+1)}{sp+q-1}}V(\l r).
\end{aligned}
\end{equation*}
Therefore, by considering small enough $\l>0$, we are justified in assuming that $V(0)>\tilde{v}(0)>0$. Define
\begin{equation*}
    R\coloneqq \mbox{sup}\{r\in(0,1)|\quad V(t)>\tilde{v}(t)\quad\mbox{in}\ (0,r)\},
\end{equation*}
and assume $R\neq 1$. Now, for all $0<r<R$ we have, where $W=U'$ and $\tilde{w}=\tilde{u}'$,
\begin{equation*}
    (Wr^{N-1})'=r^{N+a-1}V^{p}(r)>r^{N+a-1}\tilde{v}^{p}(r)=(\tilde{w}r^{N-1})'.
\end{equation*}
An integration over $[0,r]$, where $0<r\leq R$, yields $W>\tilde{w}$ on $(0,R]$. Using a similar strategy we see
\begin{equation*}
(V'r^{N-1})'=r^{N+b-1}V^{q}(r)W^{s}(r)>r^{N+b-1}\tilde{v}^{q}(r)\tilde{w}^{s}(r)=\l^{B}(\tilde{v}'(r)r^{N-1})'.
\end{equation*}
This implies that $V'>\tilde{v}'$ on $(0,R]$ and hence $V>\tilde{v}'$ on $[0,R]$, contradicting the fact that $R\neq 1$. Hence $R=1$, and so $V>\tilde{v}$ on $(0,1)$. But this then implies that $\lim\limits_{r\to 1}V(r)=\infty$, which contradicts the fact that $V(r)$ is a global solution.

\section{Proof of the Theorem \ref{thmrn}}
We obtained that $u'$, $v'$, $u$, $v$ are increasing in the proof of the Theorem \ref{thm1} and 
\begin{equation*}
\left\{
\begin{aligned}
&u'(r)=r^{1-N}\int_{0}^{r}{t^{N-1+a}v^{p}(t)}dt\quad\mbox{ for all } r>0,\\
&v'(r)=r^{1-N}\int_{0}^{r}{t^{N-1+b} v^{q}(t)(u')^{s}(t)}dt \quad\mbox{ for all } r>0,
\end{aligned}
\right.
\end{equation*}
which gives us
\begin{equation}\label{rn1}
\frac{r^{a+1}v^{p}(0)}{N+a}\leq u'(r)\leq \frac{r^{a+1}v^{p}(r)}{N+a} \quad\mbox{ for all } r>0
\end{equation}
and
\begin{equation}\label{rn2}
\frac{v^{ps+q}(0)r^{(a+1)s+b+1}}{(N+b)(N+a)^s}\leq v'(r)\leq \frac{r^{b+1}v^{q}(r)u'^{s}(r)}{N+b} \quad\mbox{ for all } r>0.
\end{equation}
Using \eqref{rn1} and \eqref{rn2} we obtain that $u'(r)$, $v'(r)$, $u(r)$, $v(r)$ tend to infinty as $r\rightarrow \infty$. Next, we do the following change of variables ( see \cite{HV1996}, \cite{BVH2010,G2012} )
$$
X(t)= \frac{ru'(r)}{u(r)},\;\;Y(t)= \frac{rv'(r)}{v(r)},\;\;Z(t)=\frac{r^{a+1}v^{p}(r)}{u'(r)} \mbox{ and }W(t)=\frac{r^{b+1}v^{q}(r)u'^{s}(r)}{v'(r)}, 
$$
where $t= \ln(r)$ for  $r\in (0,\infty)$. Direct computation shows that $(X(t),Y(t),Z(t),W(t))$ satisfies
\begin{equation}\label{rn3}
\left\{
\begin{aligned}
&X_{t}= X(Z-(N-2)-X) \quad\mbox{ for all } t\in \bR, \\
&Y_{t}= Y(W-(N-2)-Y) \quad\mbox{ for all } t\in \bR, \\
&Z_{t}= Z(N+a+pY-Z) \quad\mbox{ for all } t\in \bR, \\
&W_{t}= W(sZ+N-sN+s+b+qY-W) \quad\mbox{ for all } t\in \bR.
\end{aligned}
\right.
\end{equation}
Also, by L'Hopital's rule we deduce that $\lim_{t\rightarrow \infty}X(t)=2-N+\lim_{t\rightarrow \infty}Z(t)$. Hence, it is enough to study the last three equations of \eqref{rn3}, that is, 
\begin{equation}\label{rn4}
\left\{
\begin{aligned}
&Y_{t}= Y(W-(N-2)-Y) \quad\mbox{ for all } t\in \bR, \\
&Z_{t}= Z(N+a+pY-Z) \quad\mbox{ for all } t\in \bR, \\
&W_{t}= W(sZ+N-sN+s+b+qY-W) \quad\mbox{ for all } t\in \bR.
\end{aligned}
\right.
\end{equation}
Our system can be rewritten as
\begin{equation}\label{rn5}
\xi_{t}= h(\xi)
\end{equation}
where
$$
\xi=\left(\begin{array}{c}Y(t)\\Z(t)\\W(t)\end{array}\right) \quad\mbox{ and } \quad h(\xi)= \left(\begin{array}{c} Y(W-(N-2)-Y)\\Z(N+a+pY-Z)\\ W(sZ+N-sN+s+b+qY-W) \end{array}\right).
$$

One notices that the system \eqref{rn5} is cooperative. Therefore, the following comparison principle holds:
\begin{lemma}\label{lrn1}
Let us assume that $\xi(t)= \left(\begin{array}{c}Y(t)\\Z(t)\\W(t)\end{array}\right)$ and $\tilde{\xi}(t)= \left(\begin{array}{c}\tilde{Y}(t)\\ \tilde{Z}(t)\\ \tilde{W}(t)\end{array}\right)$ are the two nonnegative solutions of \eqref{rn5} such that
$$
Y(t_0)\geq \tilde{Y}(t_0), \;\;\;Z(t_0)\geq \tilde{Z}(t_0),\;\;\; W(t_0)\geq \tilde{W}(t_0)
$$
for some $t_0\in \bR$. Then, we have
$$
Y(t)\geq \tilde{Y}(t), \;\;\;Z(t)\geq \tilde{Z}(t),\;\;\; W(t)\geq \tilde{W}(t) \quad\mbox{ for all }t\geq t_{0}.
$$
\end{lemma}

\bigskip

\bigskip

Using \eqref{rn1} and \eqref{rn2} we deduce that $Z\geq N+a$ and $W\geq N+b$. Hence, we only have two equilibria of \eqref{rn4} which satisfy $Z\geq N+a$ and $W\geq N+b$, that is,
$$
\xi_{1}= \left(\begin{array}{c}0\\N+a\\N+s(a+1)+b\end{array}\right) \quad\mbox{  and }
\xi_{2}=\left(\begin{array}{c}2 + \frac{b+2q+s(1+a+2p)}{1-ps-q},\\N+a + p\Big(2+\frac{b+2q+s(1+a+2p)}{1-ps-q}\Big)\\N + \frac{b+2q+s(1+a+2p)}{1-ps-q}\end{array}\right).
$$
\begin{lemma}\label{lrn2}
$\xi_{2}$ is asymptotically stable.
\end{lemma}
\begin{proof}
At $\xi_2$, we have the following linearized matrix: 
$$
M=\left[\begin{array}{ccc}-Y_2&0&Y_2\\pZ_2&-Z_2&0\\qW_2&sW_2&-W_2\end{array}\right],
$$
and the associated characteristic polynomial of $M$ is
$$
P(\lambda)=\det(\lambda I-M)=\lambda^{3}+\alpha\lambda^{2}+\beta\lambda+(1-ps-q)\gamma,
$$
where
\begin{equation*}
\begin{aligned}
\alpha&=Y_2+Z_2+W_2\\
\beta&=Y_2Z_2+Z_2W_2+(1-q)Y_2W_2\\
\gamma&=Y_2Z_2W_2.
\end{aligned}
\end{equation*}
As $ps+q<1$ and  $\alpha$, $\beta$, $\gamma>0$, we have that $P(\lambda)>0$ for all $\lambda \geq 0$. In case $P$ has three real roots, then all of them are negative, which makes $\xi_2$ asymptotically stable. Next, we need to consider the case where $P$ has exactly one real root. So, let us assume that $\lambda_1 \in \bR$ and $\lambda_2,\lambda_3 \in \mathbb{C}\setminus \bR$ be the roots of characteristic polynomial  $P$. We need to show that ${\rm Re}(\lambda_2)={\rm Re}(\lambda_3)<0$, that is, $P(-\alpha)=-\beta\alpha+(1-ps-q)\gamma<0$, which is same as claiming, $\beta\alpha>(1-ps-q)\gamma$. By the use of AM-GM inequality we get that
$$
\alpha \geq 3\sqrt[3]{Y_2Z_2W_2}\quad  \mbox{ and } \quad \beta \geq 3(1-q)^{\frac{1}{3}}\sqrt[3]{(Y_2Z_2W_2)^{2}}
$$
which further gives us the desired result, that is, $\alpha\beta> (1-ps-q)\gamma$. Hence, $\xi_2$ is asymptotically stable.

\end{proof}

\begin{lemma}\label{lrn3}
For all $t\in \bR$, we have
\begin{equation}\label{rn6}
0\leq Y(t)\leq 2 + \frac{b+2q+s(1+a+2p)}{1-ps-q},
\end{equation}
\begin{equation}\label{rn7}
N+a\leq Z(t)\leq N+a + p\Big(2+\frac{b+2q+s(1+a+2p)}{1-ps-q}\Big),
\end{equation}
\begin{equation}\label{rn8}
N+s(a+1)+b\leq W(t)\leq N + \frac{b+2q+s(1+a+2p)}{1-ps-q}.
\end{equation}
\end{lemma}
\begin{proof}
As $v'(0)=0$ and $v(0)=0$, we deduce that $\lim_{t\rightarrow -\infty}Y(t)= \lim_{r\rightarrow 0}\frac{rv'(r)}{v(r)}=0$. 
	
Next, we show that there exists $t_{j}\rightarrow -\infty $ such that
\begin{equation}\label{rn9}
\left\{
\begin{aligned}
&Y(t_{j})\leq Y_2,\\
&Z(t_j)\leq Z_2,\\
&W(t_j)\leq W_2.
\end{aligned}
\right.
\end{equation}
Since $\lim_{t\rightarrow -\infty}Y(t)=0$ and $\lim_{t\rightarrow -\infty}Z(t)= N+a$, we only need to prove the last part of \eqref{rn9}. So, let us assume by contradiction that this is not true. Thus $W> W_2$ in $(-\infty,t_0)$ for some $t_0\in \bR$. Then, we have
$$
W_t= W(sZ+N-sN+s+b+qY-W)< 0 \quad\mbox{ in }(-\infty,t_0).
$$
for small enough $t_0$. Therefore, $W$ is decreasing in the neighbourhood of $-\infty$, that is, there exists $\ell= \lim_{t\rightarrow -\infty}W(t)$. Again, by the use L'Hopital's rule we obtain
\begin{equation*}
\begin{aligned}
\ell=\lim_{t\rightarrow -\infty}W(t)&=\lim_{r\rightarrow 0}\frac{r^{b+1}v^{q}(r)u'^{s}(r)}{v'(r)}\\
&=\frac{s(N+a)-s(N-1)+b+1}{1-\frac{N-1}{\ell}},
\end{aligned}
\end{equation*}
which yields $\ell=N+s(a+1)+b<W_2$ and this contradicts the assumption that $W>W_{2}$ in a neighbourhood of $-\infty$. Hence, with this we have proven the last part of \eqref{rn9}. Next, we apply the Comparison Lemma \ref{lrn1} on all the intervals $[t_j,\infty)$ for $j\geq 1$ in order to obtain the upper bound inequalities in Lemma \ref{lrn3}. Similarly, the lower bound inequalities can be obtained.
\end{proof}

\medskip

Now, assume $L=\overline{[[\xi_{1},\xi_{2}]]}\subset \bR^{3}$. We have that $\omega(\zeta)\subseteq L $ by Lemma \ref{lrn3}. As $\zeta_2$ is asymptotically stable, we get that $L$ has no circuits. Also, by \eqref{div}, one could obtain that
$$
{\rm div}\, h(\xi)=-W+(s-2)Z+(-2+p+q)Y+2+a+N(1-s)+s+b<0\quad\mbox{ in } L.
$$
By the use of Theorems \ref{thmdet2} and \ref{thmdet3} we get that $\omega(\xi)$ reduces to one of the equilibria $\xi_1$ or $\xi_2$. In case $\xi(t)\rightarrow \xi_1$ as $t\rightarrow \infty$, then we have that $Y(t)\rightarrow 0$ as $t\rightarrow \infty$. Also, we obtain that $Y_t> 0$ in a neighbourhood of infinity by using the second equation of \eqref{rn3} which is impossible given that $Y(t)> 0$ in $\bR$. Therefore, $\xi(t)\rightarrow \xi_2$ as $t\rightarrow \infty$, that is
\begin{equation*}
\begin{aligned}
\lim_{t\rightarrow \infty}X(t)&= 2+a + p\Big(2+\frac{b+2q+s(1+a+2p)}{1-ps-q}\Big),\\
\lim_{t\rightarrow \infty}Y(t)&= 2 + \frac{b+2q+s(1+a+2p)}{1-ps-q},\\
\lim_{t\rightarrow \infty}Z(t)&= N+a + p\Big(2+\frac{b+2q+s(1+a+2p)}{1-ps-q}\Big),\\
\lim_{t\rightarrow \infty}W(t)&= N + \frac{b+2q+s(1+a+2p)}{1-ps-q}.
\end{aligned}
\end{equation*}
Now , using $(X(t),Y(t),Z(t),W(t))$ , we get
\begin{equation*}
\lim_{|x|\rightarrow \infty}\frac{u(x)}{|x|^{\frac{(a+2)(1-ps-q)+ ps(a+1)+bp+2q}{1-ps-q}}}= \frac{(AB^s K)^{\frac{p}{ps + q - 1}}}{DK}
\end{equation*}
and
\begin{equation*}
\lim_{|x|\rightarrow \infty}\frac{v(x)}{|x|^{\frac{(a+1)s + b + 2}{1-ps-q}}}= (AB^s K)^{\frac{1}{ps + q - 1}},
\end{equation*}
where, 
$$
A = \lim_{t\rightarrow \infty}Y(t),
$$

$$
B = \lim_{t\rightarrow \infty}Z(t),
$$

$$
K = \lim_{t\rightarrow \infty}W(t),
$$

$$
D = \lim_{t\rightarrow \infty}X(t).
$$

\end{document}